\documentclass[english]{amsart}

\usepackage{esint}
\usepackage[svgnames]{xcolor} 
\usepackage[colorlinks,citecolor=red,pagebackref,hypertexnames=false,breaklinks]{hyperref}
\usepackage{pgf,tikz}
\usepackage{pdfsync}

\usepackage{dsfont}
\usepackage{url}
\usepackage[utf8]{inputenc}
\usepackage[T1]{fontenc}
\usepackage{lmodern}
\usepackage{babel}
\usepackage{mathtools}  
\usepackage{amssymb}
\usepackage{lipsum}
\usepackage{mathrsfs}
\usepackage{color}

\newtheorem{theorem}{Theorem}[section]
\newtheorem{proposition}{Proposition}[section]
\newtheorem{lemma}{Lemma}[section]

\newtheorem{corollary}{Corollary}[section]

\numberwithin{equation}{section}

\title[Strong uniqueness of continuation]{Quantitative strong unique continuation for elliptic operators - application to an inverse spectral problem}

\author[Mourad Choulli]{Mourad Choulli}
\address{Universit\'e de Lorraine}
\email{mourad.choulli@univ-lorraine.fr}

\date{}

\begin{document}

\frenchspacing

\begin{abstract}
Based on the three-ball inequality and the doubling inequality established in \cite{MV}, we quantify the strong unique continuation established by Koch and Tataru \cite{KT} for elliptic operators with unbounded lower-order coefficients. We also obtain a quantitative strong unique continuation for eigenfunctions, which we use to prove that two Dirichlet-Laplace-Beltrami operators are gauge equivalent when their corresponding metrics coincide in the neighborhood of the boundary and their boundary spectral data coincide on a subset of positive measure.
 \end{abstract}

\subjclass[2010]{35B60, 35J15, 35R30}

\keywords{Elliptic operators, unbounded coefficients, strong unique continuation, three-ball inequality, doubling inequality, eigenfunctions, Laplace-Beltrami operator, gauge equivalent operators.}

\maketitle


\section{Quantitative strong unique continuation}

Let $\Omega$ be a Lipschitz domain of $\mathbb{R}^n$, $n\ge 3$, with boundary $\Gamma$, and $A=(a^{k\ell})$  be a symmetric matrix with Lipschitz continuous components so that
\begin{align*}
&\kappa ^{-1}|\xi| ^2\le A(x)\xi\cdot \xi \le \kappa|\xi|^2 \quad x\in \overline{\Omega} ,\; \xi\in \mathbb{R}^n,
\\
&| a^{k\ell}(x)-a^{k\ell}(y)|\le \varkappa |x-y|,\quad x,y\in \overline{\Omega},\; 1\le k,\ell\le n,
\end{align*}
where $\kappa >1$ and $\varkappa>0$ are constants.

Let $V\in L^{n/2}(\Omega;\mathbb{R})$ and $B,W\in L^p(\Omega,\mathbb{R}^n)$, where $p>n$ is fixed, and
\[
\mathscr{E}u:=\mathrm{div}(A\nabla u+uB)+W\cdot \nabla u +Vu.
\]
The following notation will be used in the remainder of this text.
\[
\Omega_\rho=\{x\in \Omega;\; \mathrm{dist}(x,\Gamma)>4\rho\},\quad \rho>0.
\]

We say that $u\in H^1(\Omega)$ satisfies $\mathscr{E}u=0$ in $\Omega$ if
\[
\int_\Omega [-(A\nabla u+uB)\cdot \nabla v+(W\cdot \nabla u+Vu)v]dx=0,\quad v\in C_0^\infty (\Omega).
\]
Recall that $x_0\in \Omega$ is a zero of infinite order of $u\in L_{\mathrm{loc}}^2(\Omega)$ if for all $N\in \mathbb{N}$ and a sufficiently small $r>0$ we have
\[
\int_{B(x_0,r)}u^2dx=O(r^N).
\]
Koch and Tataru \cite{KT} proved that  if $u\in H^1(\Omega)$ satisfies $\mathscr{E}u=0$ in $\Omega$ and has a zero of infinite order in $\Omega$, then $u=0$. In other words, all zeros of $u\in H^1(\Omega)\setminus\{0\}$ satisfying $\mathscr{E}u=0$ in $\Omega$ are of finite order. Our objective is then to quantify the order of the zeros of the solutions of $\mathscr{E}u=0$ in $\Omega$. Let $\Omega_0\Subset \Omega$. Let $\mathscr{C}$  be the set of all positive
constants that depend on $n$, $\Omega$, $\Omega_0$, $\kappa$, $\varkappa$, $B$, $W$ and $V$. When $\aleph$ is an arbitrary quantity, we write $\mathbf{c}\in \mathscr{C}_\aleph$ to denote that $\mathbf{c}$ is a positive constant depending on $n$, $\Omega$, $\Omega_0$, $\kappa$, $\varkappa$, $B$, $W$, $V$ and $\aleph$. Set
\[
\mathscr{S}_0=\left\{ u\in H^1(\Omega)\setminus\{0\};\;  \mathscr{E}u=0\; \mathrm{in}\; \Omega\right\}.
\]

\subsection{Statement of the main theorem}

Recall that, according to \cite[Corollary 4.3]{ER}, there exists $\rho_0=\rho_0(\Omega)>0$ such that $\Omega_\rho$ is connected for all $0<\rho<\rho_0$.

\begin{theorem}\label{theorem0}
There exist $\rho_\ast \in \mathscr{C}$, $0<\rho_\ast \le \rho_0$, $\tau\in \mathscr{C}$ with $\tau<1/4$ so that, for all $0<\rho <\rho_\ast$, we find $C<1$, $c>1$ and $\mathfrak{a}>1$ in $\mathscr{C}_\rho$ with the property that if $u\in \mathscr{S}_0$ then
\begin{equation}\label{1}
Cr^{c}\left(\frac{\|u\|_{L^2(\Omega_0)}}{\|u\|_{L^2(\Omega)}}\right)^{\mathfrak{b}_r}\le  \frac{\|u\|_{L^2(B(x,r))}}{\|u\|_{L^2(\Omega)}},\quad x\in \Omega_{\rho},\; 0<r<\tau \rho,
\end{equation}
where $\mathfrak{b}_r=\ln (\tau\rho/r)+\mathfrak{a}$.
\end{theorem}

With the assumptions and the notations of theorem \ref{theorem0}, if 
\[
q(u)=\left|\ln \left(\frac{\|u\|_{L^2(\Omega_0)}}{\|u\|_{L^2(\Omega)}}\right)\right|
\]
then \eqref{1} gives for $x\in \Omega_{\rho}$ and  $0<r<\tau \rho$
\[
\left[C(\tau \rho)^{-q(u)}e^{-\mathfrak{a}q(u)}\|u\|_{L^2(\Omega)}\right]r^{c+q(u)}\le \|u\|_{L^2(B(x,r))} .
\]
This inequality clearly quantifies strong unique continuation.

Davey and Zhu \cite{DZ} considered the case $B=0$, $W\in L^p$, $p>n$ and $V\in L^q$, $q>n/2$. The inequalities established in \cite{DZ} are with constants depending explicitly  on the norms of $W$ and $V$. The result in \cite{DZ} was improved in \cite{Da} in the case $B=0$, $W=0$ and $V\in L^q$, $q>n/2$.

As a by-product of theorem \ref{theorem0}, we obtain the following interpolation inequality.

\begin{corollary}\label{corollary0}
Let $\omega\Subset\Omega$. Then there exist $C>0$ and $0<t<1$ in $\mathscr{C}_\omega$ such that for any $u\in \mathscr{S}_0$ we have
\begin{equation}\label{ii0}
C\|u\|_{L^2(\Omega_0)}\le \|u\|_{L^2(\Omega)}^{1-t}\|u\|_{L^2(\omega)}^t.
\end{equation}
\end{corollary}

Henceforth, $\tau$ is as in Theorem \ref{theorem0}.

\subsection{Quantitative unique continuation from a set of positive measure}

Let
\[
\mathscr{S}_1=\left\{ u\in H^1(\Omega);\; \|u\|_{L^2(\Omega)}= 1\; \mathrm{and}\; \mathscr{E}u=0\; \mathrm{in}\; \Omega\right\}.
\]

\begin{proposition}\label{proposition1}
$($\cite[Proposition 4]{MV}$)$ Let $m>0$, $\rho_1>0$ and $0<\sigma<1/n$. There exists $\epsilon_0=\epsilon_0(m,\rho,\sigma)$ such that if $u\in \mathscr{S}_1$ and $E\subset \Omega_{\rho_1}$ verify $|E|> m$ and $\||E|^{-1/2}u\|_{L^2(E)}<\epsilon_0$ then we find $B(z,\tau\rho_1)$ with $z\in \Omega_{\rho_1}$ for which we have
\begin{equation}\label{2}
\|u\|_{L^2(B(z,\tau\rho_1))}\le |\ln \||E|^{-1/2}u\|_{L^2(E)}|^{-\zeta},
\end{equation}
where $\zeta\in \mathscr{C}$.
\end{proposition}

Under the assumptions and the notations of Proposition \ref{proposition1}, we apply \eqref{ii0} with $u\in \mathscr{S}_1$ and $\omega=B(z,\tau\rho_1)$ to obtain that if $\||E|^{-1/2}u\|_{L^2(E)}<\epsilon_0$ then we have the following variant of the inequality  of \cite[Theorem 1]{MV}. 
\begin{equation}\label{ms}
\|u\|_{L^2(\Omega_0)}\le C\left|\ln \||E|^{-1/2}u\|_{L^2(E)}\right|^{-t\zeta},
\end{equation}
where $C\in \mathscr{C}_{\rho_1,m,\sigma}$ with $\rho_1$, $m$ and $\sigma$ being as in Proposition \ref{proposition1},  and $t$ is as in Corollary \ref{corollary0}. Note that the original inequality in \cite[Theorem 1]{MV} corresponds to \eqref{ms} in which instead of the left-hand side term we have $\|u\|_{L^2(\Omega_\rho)}$.

\subsection{Strong unique continuation for eigenfunctions}

Let $\lambda > 0$ and $u\in H^1(\Omega)$ satisfies $(\mathscr{E}+\lambda )u=0$ in $\Omega$. Then $v=e^{\sqrt{\lambda}t}u$ is a solution of the equation
\[
(\mathscr{E}+\partial _t^2)v=0\quad  \mathrm{in}\; \Omega \times (0,1). 
\]
By applying \eqref{1} with $\Omega$, $\Omega_0$ and $\mathscr{E}$ replaced  by $\Omega \times (0,1)$, $\Omega_0\times (1/4,3/4)$) and $\mathscr{E}+\partial _t^2$, respectively, we obtain
\begin{align}
&Cr^c\left(\frac{\|v\|_{L^2(\Omega_0\times (1/4,3/4))}}{\|v\|_{L^2(\Omega\times (0,1))}}\right)^{\mathfrak{b}_r}\le \frac{\|v\|_{L^2(B((x,t),r))}}{\|v\|_{L^2(\Omega\times (0,1))}}\label{4}
\\
&\hskip 5.5cm \quad  (x,t)\in (\Omega\times(0,1))_\rho,\; 0<r<\tau \rho.\nonumber
\end{align}
Here and henceforth, $C<1$ $c>1$ are generic constants belonging to $\mathscr{C}_\rho$ and $\mathfrak{b}_r$ is the same as in Theorem \ref{theorem0} when $\Omega$ and $\Omega_0$ are replaced  by $\Omega \times (0,1)$ and $\Omega_0\times (1/4,3/4)$, respectively.

Next, assume that $\rho<\min (1/8,\rho_0)$, where $\rho_0$ is as in Theorem \ref{theorem0}. Let $(x,t)\in \Omega_\rho \times (4\rho,1-4\rho)\subset (\Omega\times(0,1))_\rho$.  As $B((x,t),r)\subset B(x,r)\times (0,1)$, \eqref{4} implies
\begin{equation}\label{4.1}
Cr^c\left(\frac{\|v\|_{L^2(\Omega_0\times (1/4,3/4))}}{\|v\|_{L^2(\Omega\times (0,1))}}\right)^{\mathfrak{b}_r}\le \frac{\|v\|_{L^2(B(x,r)\times (0,1))}}{\|v\|_{L^2(\Omega\times (0,1))}},\quad   x\in \Omega_\rho,\; 0<r<\tau \rho.
\end{equation}
Hence
\begin{equation}\label{5}
Cr^ce^{-(3\mathfrak{b}_r/4)\sqrt{\lambda}-1}\left(\frac{\|u\|_{L^2(\Omega_0)}}{\|u\|_{L^2(\Omega)}}\right)^{\mathfrak{b}_r}\le \frac{\|u\|_{L^2(B(x,r))}}{\|u\|_{L^2(\Omega)}},\quad  x\in \Omega_\rho,\; 0<r<\tau \rho.
\end{equation}

Let $\omega \Subset \Omega$. The following inequality then follows easily from \eqref{ii0}
\begin{equation}\label{6}
\|u\|_{L^2(\Omega_0)}\le Ce^{3\sqrt{\lambda}/4}\|u\|_{L^2(\Omega)}^{1-t} \|u\|_{L^2(\omega)}^t,
\end{equation}
where $C>0$ and $0<t<1$ belong to $\mathscr{C}_\omega$.

This type of estimate is well known for an eigenfunction $u$ of the Laplace-Beltrami operator on a compact Riemannian manifold $M$ without boundary or a compact Riemannian manifold $M$ with boundary and $u\in H_0^1(M)$. More precisely, we have an estimate of the form
\[
\|u\|_{L^2(M)}\le C' e^{c'\sqrt{\lambda}}\|u\|_{L^2(\omega)},
\]
where the constants $C'>0$ and $c'>0$ only depend of $\mathrm{dim}(M)$, $\omega$ and the metric on $M$ (e.g. \cite[Theorem 1.10]{LL} and the references therein).

Now suppose that $\rho_1$, given by proposition \ref{proposition1} and corresponding to $\Omega\times (0,1)$ instead of $\Omega$, is small enough so that $E\times J\subset (\Omega\times (1/4,3/4))_{\rho_1}$ and $|E|>2m$ (hence $|E||J|>m$). Moreover, assume $\|u\|_{L^2(\Omega)}=1$ and let $v=\varsigma_\lambda e^{\sqrt{\lambda}t}u$, where $\varsigma_\lambda=\|e^{\sqrt{\lambda}t}\|_{L^2((0,1))}^{-1}$. Then it follows from \eqref{ms} that
\begin{equation}\label{3.1}
C\|v\|_{L^2(\Omega_0\times (1/4,3/4))}
\le  \left|\ln\left( \||E\times J|^{-1/2}v\|_{L^2(E\times J)}\right)\right|^{-t\zeta},
\end{equation}
where $ C\in \mathscr{C}_{\rho_1,m,\sigma}$ with $\rho_1$, $m$ and $\sigma$ being as in Proposition \ref{proposition1} when $\Omega$ and $\Omega_0$ are replaced by $\Omega\times (0,1)$ and $\Omega_0\times (1/4,3/4)$, respectively  and $t$ and $\zeta$ are as in \eqref{ms} with $\Omega$ and $\Omega_0$ are replaced by $\Omega \times (0,1)$ and $\Omega_0\times (1/4,3/4)$, respectively. Therefore, we have
\begin{equation}\label{7}
C\|u\|_{L^2(\Omega_0)} \le \varsigma_\lambda e^{3\sqrt{\lambda}/4} \left|\ln\left( \varsigma_\lambda' e^{-\sqrt{\lambda}}\|(2|E|)^{-1/2}u\|_{L^2(E)}\right)\right|^{-t\zeta},
\end{equation}
where $C$, $t$ and $\zeta$ are the same as in \eqref{3.1}, and $\varsigma_\lambda'=\varsigma_\lambda\|e^{\sqrt{\lambda}}t\|_{L^2(J)}$.

Next, consider the bilinear form
\[
\mathfrak{B}(u,v)=\int_\Omega [(A\nabla u+uB)\cdot \nabla v-(W\cdot \nabla u+Vu)v]dx,\quad u,v\in H_0^1(\Omega).
\]
We can proceed as \cite[Subsection 1.3]{Ch} to prove that $\mathfrak{a}$ is continuous and coercive. Therefore, the operator $\mathfrak{A}:H_0^1(\Omega)\rightarrow H^{-1}(\Omega)$ defined by
\[
\langle \mathfrak{A}u,v\rangle =\mathfrak{B}(u,v),\quad u,v\in H_0^1(\Omega),
\]
is bounded, where $\langle \cdot ,\cdot \rangle$ is the duality pairing between $H_0^1(\Omega)$ and $H^{-1}(\Omega)$.

Under the hypothesis $W=B$, $\mathfrak{A}$ is self-adjoint and therefore, by \cite[Theorem 3.37, page 49]{Mc}, $\mathfrak{A}$ is diagonalizable, which means that the spectrum of $\mathfrak{A}$ consists of a nondecreasing  real-valued sequence $(\lambda_j)$ converging to $\infty$:
\[
-\infty <\lambda_1\le \lambda_2\le \ldots \le \lambda_j \le \ldots
\]  
and the exists an orthonormal basis $(\phi_j)$ of $L^2(\Omega)$ consisting of eigenfunctions. That is,  for all $j\ge 1$
\[
\mathfrak{B}(\phi_j,v)=\lambda_j(\phi_j|v),\quad u,v\in H_0^1(\Omega),
\] 
where $(\cdot|\cdot)$ is the usual scalar product of $L^2(\Omega)$.

Fix $\lambda^\ast\in \mathbb{R}$  so that $\lambda_1+\lambda^\ast >0$. By taking in \eqref{5}, \eqref{6} and \eqref{7}, in which we substitute $\mathscr{E}$ by $\mathscr{E}-\lambda^\ast$ and $u=\phi_j$, $j\ge 1$, we obtain, where for simplicity $\lambda_j+\lambda^\ast$ is replaced by $\lambda_j$,
\begin{align*}
&Cr^ce^{-3\sqrt{\lambda_j}\mathfrak{b}_r/4-1}\|\phi_j\|_{L^2(\Omega_0)}^{\mathfrak{b}_r}\le \|\phi_j\|_{L^2(B(x,r))},\quad  x\in \Omega_\rho,\; 0<r<\tau \rho,\; j\ge 1,
\\
&\|\phi_j\|_{L^2(\Omega_0)}\le Ce^{3\sqrt{\lambda_j}/4}\|\phi_j\|_{L^2(\omega)}^t,\quad j\ge 1,
\\
& C\|u\|_{L^2(\Omega_0)}\le \varsigma_\lambda e^{3\sqrt{\lambda_j}/4} \left|\ln\left( \varsigma_\lambda' e^{-\sqrt{\lambda}}\|(2|E|)^{-1/2}u\|_{L^2(E)}\right)\right|^{-t\zeta},\quad j\ge 1,
\end{align*}
where the above constants are the same as in \eqref{5}, \eqref{6} and \eqref{7}, respectively, and do not depend on $j$.

\subsection{Determining a metric tensor from a partial spectral boundary data}

In this subsection, $\Omega$ is of class $C^\infty$ and $g=(g_{k\ell})$ is a $C^\infty$  Riemannian metric on $\overline{\Omega}$. Denote by $(g^{k\ell})$ the inverse matrix of $(g_{k\ell})$ and recall that the Laplace-Beltrami operator associated to $g$ is given by
\[
\Delta_g=\frac{1}{\sqrt{|g|}}\sum_{k,\ell=1}^n\frac{\partial}{\partial x_k}\left(\sqrt{|g|}g^{k\ell}\frac{\partial}{\partial x_\ell}\, \cdot \right),
\]
where $|g|$ denotes the determinant of $g$.

Let $dV=\sqrt{|g|}dx^1\ldots dx^n$ be the volume form and define the unbounded operator 
\[
A_g:L^2(\Omega,dV)\rightarrow L^2(\Omega,dV)
\]  
acting as follows
\[
A_g=-\Delta _g\quad \mbox{with}\quad D\left(A_g\right)=H_0^1(\Omega)\cap H^2(\Omega).
\]
As $A_g$ is self-adjoint operator with compact resolvent, its spectrum is reduced to a sequence of eigenvalues:
\[
0 < \lambda_1^g<\lambda_2^g\le  \ldots \lambda_k^g\le \ldots \quad \mbox{and}\quad \lambda_k^g\rightarrow \infty \; \mbox{as}\; k\rightarrow \infty.
\]

Moreover, there exists $(\phi_k^g)$, $k\ge 1$, an orthonormal basis of $L^2(\Omega,dV)$ consisting of eigenfunctions, each $\phi_k^g$ being an eigenfunction for $\lambda_k^g$, that is for all $k\ge 1$, $\phi_k^g\in H_0^1(\Omega)\cap H^2(\Omega)$ and
\[
-\Delta_g \phi_k^g=\lambda_k^g\phi_k^g.
\]
Note that, according to the elliptic regularity, $\phi_k^g\in C^\infty(\overline{\Omega})$ for all $k\ge 1$.

For convenience, we use the following notation 
\[
\psi_k^g=\partial_\nu  \phi_k^g, \quad k\ge 1,
\]
where $\nu$ is the unit normal exterior vector field on $\Gamma$ with respect to $g$.

For $\chi\in C^\infty (\overline{\Omega})$ so that $\chi>0$, define the operator 
\[
A_g^\chi:L^2(\Omega,\chi^{-2}dV)\rightarrow L^2(\Omega,\chi^{-2}dV)
\]
 by
\[
A_g^\chi=\chi A_g\chi^{-1},\quad D(A_g^\chi)=D(A_g).
\]

In the following, $g_1$ and $g_2$ denote two metric tensors on $\overline{\Omega}$. We say that the operators $A_{g_1}$ and $A_{g_2}$ are gauge equivalent if $A_{g_2}=A_{g_1}^\chi$, for some $\chi \in C^\infty(\overline{\Omega})$ verifying $\chi>0$. It is important to note that if $A_{g_1}$ and $A_{g_2}$ are gauge equivalent then $(g_1^{k\ell})=(g_2^{k\ell})$ (see \cite[2.2.9]{KKL}). Moreover, if $A_{g_1}$ and $A_{g_2}$ are gauge equivalent then $(\lambda^{g_2}_k)= (\lambda^{g_1}_k)$ (\cite[(2.54)]{KKL}) and $(\phi^{g_2}_k)= (\chi\phi^{g_1}_k)$ (\cite[(2.55)]{KKL}).
 
The sequence $(\lambda^g_k,\psi^g_k)$ will be called in the following the spectral boundary data of $A_g$.

Fix $\tilde{x}\in \Gamma=\partial \Omega$ and let $\Sigma$ be a measurable subset of $\Gamma$ so that
\[
|\Sigma \cap B(r)|>0,\quad 0<r\le r_0,
\]
for some $r_0>0$, where $B(r)=B(\tilde{x},r)$. Assume in addition that $\Omega$ is chosen so that $\Gamma$ has a (smooth) connected neighborhood $N$ in $\overline{\Omega}$.

Pick $\lambda >0$ and let $u\in W^{2,\infty}(N)$ be a solution of the following BVP
\begin{equation}\label{bvp1}
(\Delta_g+\lambda )u=0\; \mathrm{in}\; N\quad u=0\; \mathrm{on}\; \Gamma.
\end{equation}

Paraphrasing the proof \cite[Theorem 2.3]{BZ}, we find that there exists $r_1\le r_0$ such that 
\begin{align}
\sup_{B(r_1/2)\cap N}&(|u|+|\nabla u|)  \label{a6}
\\
&\le Ce^{\lambda r_1}\left(\sup_{B(r_1)\cap \Sigma}|\partial_\nu u|\right)^\alpha \left(\sup_{B(r_1)\cap N}(|u|+|\nabla u|)\right)^{1-\alpha},\nonumber
\end{align}
where the constants $C>0$ and $0<\alpha <1$ only depends on $n$, $N$ and $\Sigma$.

Let $N_0\Subset B(r_1/2)\cap N$ be arbitrarily fixed. Then \eqref{a6} together with \eqref{6} imply
\begin{equation}\label{a7}
\|u\|_{L^2(N_0)}  \le Ce^{\lambda r_1+3\sqrt{\lambda}/4}\|\partial_\nu u|\|_{L^\infty(\Sigma)}^\alpha \|u\|_{W^{1,\infty}(N)}^{1-\alpha}.
\end{equation}

The partial spectral boundary  data of $A_g$  will consist of $(\lambda^g_k,\psi^g_k{_{|\Sigma}})$.

\begin{theorem}\label{theorem2}
Assume that $g_1=g_2$ in $N$. If $A_{g_1}$ and $A_{g_2}$ have the same partial boundary spectral data then $A_{g_1}$ and $A_{g_2}$ are gauge equivalent.
\end{theorem}

\begin{proof}
For all $k\ge 1$, we verify that $u_k=\phi_k^{g_1}-\phi_k^{g_2}$ is a solution of \eqref{bvp1} with $\lambda=\lambda_k^{g_1}$. We then apply \eqref{a7} to deduce that $u_k=0$ in $N_0$ and therefore $u=0$ in $N$ by the unique continuation for $\Delta_g+\lambda_k$. We have in particular $\psi_k^{g_1}=\psi_k^{g_2}$ and therefore $A_{g_1}$ and $A_{g_2}$ admit the same spectral boundary data. We invoke \cite[Theorem 3.3]{KKL} to complete the proof.
\end{proof}

We emphasize that when $\Sigma$ is an arbitrary non-empty open set of $\Gamma$, a logarithmic stability inequality has recently been established in \cite{IY} in the two-dimensional case. The problem of determining the potential in a Schrödinger operator or the potential and magnetic field appearing in a magnetic Schrödinger operator has been widely studied. We cite only a few recent works \cite{AS,Bel87,Bel92,BCDKS,BCY,BCY2,BD,ChS,Is,KK,KKL,KKS,Ki1,KOM,NSU,Po}.

\section{Proof of Theorem \ref{theorem0}}

The following  special case of \cite[Theorem 3]{MV} will be used in the proof of Theorem \ref{theorem0}.

\begin{theorem}\label{theorem1}
(three-ball inequality) There exist $0<\alpha<1$, $\tau <1/4$, $\rho_\ast\in \mathscr{C}$, $\rho_\ast\le \rho_0$, and $C\in \mathscr{C}$ such that for all $0<\rho<\rho_\ast$, $x\in \Omega_\rho$, $r=\tau \rho/4$ and $u\in H^1(\Omega)$ satisfying $\mathscr{E}u=0$ in $\Omega$ we have
\begin{equation}\label{tbi1}
\|u\|_{L^2(B(x,2r))}\le C\|u\|_{L^2(B(x,4r))}^{1-\alpha}\|u\|_{L^2(B(x,r))}^\alpha.
\end{equation}
\end{theorem}

\begin{proof}[Proof of Theorem \ref{theorem0}]
In this proof $C\ge 1$ and $0<\alpha<1$ are generic constants belonging to $\mathscr{C}$.

Let $\rho<\rho_\ast$, where $\rho_\ast$ is as in Theorem \ref{theorem1}, $r=\tau\rho/4$. 
Let $Q$ be the smallest closed cube containing $\overline{\Omega}$. Then we have $|Q|=d^n$. We divide $Q$ into $m_\rho$ closed sub-cubes. Let $(Q_j)_{1\le j\le m_\rho}$ denotes the family of these cubes. Note that, for each $j$, $|Q_j|< (\rho/\sqrt{n})^n$ and therefore $Q_j$ is contained in a ball $B_j$ of radius $\rho/2$. Define
\[
I_\rho=\{j \in \{1,\ldots,m_r\};\; Q_j\cap \overline{\Omega}^{\rho}\ne \emptyset\}\quad \mathrm{and}\quad Q^\rho=\bigcup_{j\in I_r}Q_j.
\]
In particular, $Q^\rho\subset \Omega^{3\rho/4}$ and since $\overline{\Omega}^{\rho}$ is connected then so is $Q^\rho$.

It follows from Lemma \ref{lemGeo} in Appendix \ref{GL} that there exists  a continuous path $\psi:[0,1]\rightarrow Q^\rho$ joining $x$ to $y$ whose length, denoted hereinafter by $\ell(\psi)$, does not exceed $\rho m_\rho$.

Let $t_0=0$ and define the sequence $(t_k)$ as follows
\[
t_{k+1}=\inf\{t\in [t_k,1];\; \psi(t)\not\in B(\psi(t_k),\rho)\},\quad k\ge 0.
\]
Then $|\psi(t_{k+1})-\psi(t_k)|=\rho$. Thus, there exists a positive integer $N_\rho$ so that $\psi(1)\in B(\psi(t_{N_\rho}),\rho)$. As $\rho N_\rho\le \ell(\psi)\le \rho m_\rho$, we obtain $N_\rho\le m_\rho$.

Let $x_j=\psi(t_k)$, $j=0, \ldots,N_\rho$ and $x_{N_\rho+1}=y$. Clearly, $B(x_j,3\rho)\subset \Omega$, $j=0, \ldots,N_\rho+1$, and  $B(x_{j+1},\rho)\subset B(x_j,2\rho)$, $j=0, \ldots,N_\rho$.

We verify that
\[
N:=N_\rho\le N_0:=\mathfrak{c}/r^n.
\]
Here and henceforth, $\mathfrak{c}=\mathfrak{c}(n,d,\rho_0)>0$ is a generic constant.

Let $u\in \mathscr{S}_0$. It follows from \eqref{tbi1}
\begin{equation}\label{est1}
\|u\|_{L^2(B(x_j,2r ))}\le C\|u\|_{L^2(\Omega)}^{1-\alpha}\|u\|_{L^2(B(x_j,r))}^\alpha,\quad 0\le j\le N+1.
\end{equation}

Let $I_j=\|v\|_{L^2(B(x_j,r ))}$, $0\le j\le N+1$, where $v=u/ \|u\|_{L^2(\Omega)}$. Since $B(x_{j+1},r )\subset B(x_j,2r)$, $1\le j\le N$, estimate \eqref{est1} implies
\begin{equation}\label{est2}
I_{j+1}\le C I_j^\alpha,\quad 0\le j\le N.
\end{equation}
As $C^{1+\alpha+\ldots +\alpha^{N+1}}\le C^{1/(1-\alpha)}$, an induction argument shows that \eqref{est2} yields
\begin{equation}\label{est3}
I_{N+1}\le C^{1/(1-\alpha)} I_0^{\alpha^{N+1}} ,
\end{equation}
which is combined with $I_0\le 1$, gives
\begin{equation}\label{est4}
 \|v\|_{L^2(B(y,r ))}\le C\|v\|_{L^2(B(x,r))}^{\alpha^{\mathfrak{c}/r^n}},
\end{equation}
where we used that $N+1\le \mathfrak{c}/r^n$.

Replacing $\rho_\ast$ by $2\rho_\ast$, we can rewrite \eqref{est4}  in the following form
\begin{equation}\label{est5}
 \|v\|_{L^2(B(y,\tau\rho/8))}\le C\|v\|_{L^2(B(x,\tau\rho/8))}^{\alpha^{\mathfrak{c}/\rho^n}},\quad x,y\in \Omega_{\rho/2}.
\end{equation}

Furthermore, by reducing $\rho_\ast$ if necessary, we assume that $\Omega_{(1+\tau/16)\rho/2}\Supset \Omega_0$. Let $M:=\|v\|_{L^2(\Omega_0)}=\|u\|_{L^2(\Omega_0)}/\|u\|_{L^2(\Omega)}$ ($<1$). Then we have
\[
M\le \|v\|_{L^2(\Omega_{(1+\tau/16)\rho/2})}^2=\sum_j\|v\|_{L^2(Q_j\cap \Omega_{(1+\tau/16)\rho/2})}^2.
\]
where $Q_j$ are defined as before so that $|Q_j|=[\tau\rho/(8\sqrt{n})]^n$. Hence, there exists $Q_k$ so that
\[
\|v\|_{L^2(Q_k)}\ge \mathfrak{c}M\rho^{-n},
\]
As $Q_k\subset B(z,\tau\rho/8)$, for some $z\in \Omega_{\rho/2}$, we obtain
\begin{equation}\label{est6}
\|v\|_{B(z,\tau\rho/8)}\ge   \mathfrak{c}M\rho^{-n}.
\end{equation}

In the rest of this proof $C_\rho\in \mathscr{C}_\rho$ is a generic constant. Let $\mathfrak{a}=1/\alpha^{\mathfrak{c}/\rho^n}$. In light of \eqref{est5}, we obtain from \eqref{est6}

\begin{equation}\label{est7}
C_\rho M^{\mathfrak{a}}\le \|v\|_{L^2(B(x,\tau\rho/8))},\quad x\in \Omega_{\rho/2}.
\end{equation}

Let
\[
\mathscr{Q}_\rho (v)=\max_{x\in \Omega_{\rho/2}}\left(\frac{\|v\|_{L^2(B(x,\rho))}}{ \|v\|_{L^2(B(x,2\tau\rho))}}\right)^\beta.
\]
Then \eqref{est7} implies
\[
\mathscr{Q}_\rho (v)\le C_\rho M^{-\beta\mathfrak{a}}.
\]

From now on, we suppose that $C_\rho>1$. By applying the following doubling inequality (\cite[Proposition 2]{MV}) 
\[
\|v\|_{L^2(B(x,2s))}\le C_\rho \mathscr{Q}_\rho (v)\|v\|_{L^2(B(x,s))},\quad x\in \Omega_\rho,\; 0<s<\tau \rho,
\]
 we deduce
\begin{equation}\label{est8}
\|v\|_{L^2(B(x,2s))}\le C_\rho M^{-\beta\mathfrak{a}}\|v\|_{L^2(B(x,s))},\quad x\in \Omega_\rho,\; 0<s<\tau \rho.
\end{equation}

Let $m$ be the nonnegative  integer satisfying $2^{m-1}s<\tau \rho$ and $2^ms\ge \tau \rho$. Iterating \eqref{est8} $(m-1)$ 
times and using that $B(x,\tau \rho)\subset B(x,2^ms)$, we obtain
\begin{equation}\label{est9}
\|v\|_{L^2(B(x,\tau \rho))}\le C_\rho^m M^{-m\beta\mathfrak{a}}\|v\|_{L^2(B(x,s))},\quad x\in \Omega_\rho,\; 0<s<\tau \rho.
\end{equation}
By noting that $\Omega_{\rho}=\Omega_{(2\rho)/2}$, an combination of \eqref{est7} and \eqref{est9} imply
\begin{equation}\label{est9}
C_\rho\le C_\rho^m M^{-m\beta\mathfrak{a}+\mathfrak{a}}\|v\|_{L^2(B(x,s))},\quad x\in \Omega_\rho,\; 0<s<\tau \rho,
\end{equation}
Replacing $\mathfrak{a}+\mathfrak{a}\beta$ by $\mathfrak{a}$ and using that $C_\rho>1$ and $M^{-1}>1$, we obtain from \eqref{est9}
\begin{equation}\label{est10}
C_\rho^{-1}\le  s^{-c_\rho}M^{-(\ln (\tau\rho/s)+\mathfrak{a}}\|v\|_{L^2(B(x,s))},\quad x\in \Omega_\rho,\; 0<s<\tau \rho,
\end{equation}
where $c_\rho>1$ belongs to $\mathscr{C}_\rho$.
In other words, we have
\[
C_\rho^{-1}\|u\|_{L^2(\Omega_0)}^{\ln (\tau\rho/s)+\mathfrak{a}}\le  s^{-c_\rho}\|u\|_{L^2(\Omega)}^{\ln (\tau\rho/s)+\mathfrak{a}-1}\|u\|_{L^2(B(x,s))},\quad x\in \Omega_{\rho},\; 0<s<\tau \rho.
\]
 The proof is then complete.
\end{proof}

\appendix

\section{Geometric lemma}\label{GL}

\begin{lemma}\label{lemGeo}
Let $Q=\cup_{j \in J} Q_j$ be the connected union of $\ell \, (=|J|)$ closed cubes with edges of length $r/\sqrt{n}$, for some $r>0$. All $x,y\in Q$ can be connected by a broken line of length less than or equal to $\ell r$.
\end{lemma}

\begin{proof}
Let $V$ denotes the union of the vertices of $\{Q_j;\; j\in J\}$. Let $x,y\in Q$ be with $x\in Q_j$ and $y\in Q_k$, for some $j,k\in J$. Let $x_v\in V\cap Q_j$ and $y_v\in V\cap Q_k$ be fixed arbitrarily. There exists in $Q$ a broken line connecting $x_v$ to $y_v$ consisting of $m$ segments $[v_j,v_{j+1}]$ such that $m\le \ell-2$, $v_0=x_v$ and $v_m=y_v$ (note that $]v_j,v_{j+1}[$ is contained inside a cube, on a face or an edge). Since $|[x,x_v]|\le r$, $|[y-y_v]|\le r$ and $|[v_j,v_{j+1}]|\le r$, $j=0,\ldots m-1$, we conclude that any two points of $Q$ can be connected by a broken line of length less than or equal to $\ell r$.
\end{proof}

\vskip .5cm
\end{document}